\documentclass[reqno]{amsart}

\makeatletter
\@addtoreset{equation}{section}

\makeatother

\newcommand{\Rm}{\mathbb{R}}

\newcommand{\be}{\begin{equation}}
\newcommand{\ee}{\end{equation}}
\newcommand{\al}[1]{\begin{align}#1\end{align}}
\newcommand{\eq}[1]{\begin{align*}#1\end{align*}}
\newcommand{\va}{\varphi}
\newcommand{\pp}{\partial}

\usepackage{amsmath}
\usepackage[foot]{amsaddr}
\usepackage{amsthm,amssymb,mathrsfs}

\newtheorem{thm}{Theorem}[section]
\newtheorem{lem}[thm]{Lemma}
\newtheorem{cor}[thm]{Corollary}
\newtheorem{prop}[thm]{Proposition}

\theoremstyle{remark}
\newtheorem{rmk}[thm]{Remark}

\title[]{Global Lipschitz stability for inverse problems for 
radiative transport equations}

\author[]{Manabu Machida $^1$}
\address{$^1$ Institute for Medical Photonics Research, Hamamatsu University School of Medicine, Hamamatsu, Shizuoka 431-3192, Japan}
\email{machida@hama-med.ac.jp}

\author[]{Masahiro Yamamoto $^{2,3,4}$}
\address{$^2$ Graduate School of Mathematical Sciences, The University of Tokyo, 3-8-1 Komaba, Meguro, Tokyo 153, Japan}
\address{$^3$ Honorary Member of Academy of Romanian Scientists, Splaiul Independentei Street, no 54, 050094 Bucharest Romania}
\address{$^4$ Peoples' Friendship University of Russia (RUDN University) 6 Miklukho-Maklaya St, Moscow, 117198, Russian Federation}
\email{myama@ms.u-tokyo.ac.jp}

\begin{document}

\begin{abstract}
We consider inverse problems of determining coefficients or time independent 
factors of source terms in radiative transport equations 
by means of Carleman estimate.
We establish global Lipschitz stability results with an additional 
condition which requires some strict positivity for
initial value or given factor of source,
but we need not any extra conditions on domains of velocities, which is the 
main achievement of this article compared with the existing work 
by Machida and Yamamoto ({\it Inverse Problems} {\bf 30} 035010, 2014). 
The proof relies on a Carleman estimate with a piecewise linear weight 
function according to the partition of the velocity domain.
\end{abstract}

\maketitle

\section{Introduction}

Let $\Omega$ be a bounded domain of $\mathbb{R}^n$, $n\ge 2$ with $C^1$-boundary $\partial\Omega$. Let $V$ be a domain in $\Rm^n$ with $0\notin\overline{V}$, where $\overline{V}$ is the closure of $V$. 
We use symbols $x = (x_1, ..., x_n) \in \mathbb{R}^n$,
$\nabla={^t}(\frac{\pp}{\pp x_1},\dots,\frac{\pp}{\pp x_n})$. 
Moreover, $v\cdot v'$ denotes the 
scalar product of vectors $v,v'$. Let $u(x,v,t)$ be the solution to the following radiative transport equation.
\al{
P_0u+\sigma(x,v)u-\int_Vk(x,v,v')u(x,v',t)\,dv'= F(x,v,t),
&\quad x\in\Omega,\; v\in V,\; 0<t<T,
\label{1.1}
\\
u(x,v,0)=a(x,v),
&\quad x\in\Omega,\;v\in V,
\label{rte0ic}
\\
u(x,v,t)=g(x,v,t)
&\quad\mbox{on}\;\Gamma_-\times(0,T),
\label{rte0bc}
}
where $F$ is a source term and 
\eq{
P_0u(x,v,t)=\pp_tu(x,v,t)+v\cdot\nabla u(x,v,t).  
}
Let $\nu(x)$ be the outward unit vector normal to $\partial\Omega$ at $x\in\partial\Omega$. We define $\Gamma_+$ and $\Gamma_-$ by
\eq{
\Gamma_{\pm}
= \left\{(x,v)\in\partial\Omega\times V;\;(\pm\nu(x)\cdot v) > 0\right\}.
}
The coefficients are assumed as follows.
\eq{
\sigma\in L^{\infty}(\Omega\times V),\quad
k\in L^{\infty}(\Omega\times V\times V).
}
In this article,
we will consider the inverse problem of determining $\sigma$ and 
a time independent factor of the source term $F(x,v,t)$
by $u$ on $\Gamma_+$, $0<t<T$, assuming that $k$ is known.

For an arbitrary fixed constant $M>0$, we set
\eq{
\mathcal{U}=\left\{u\in X;\;
\|u\|_X+\|\nabla u\|_{H^1(0,T;L^2(\Omega\times V))}\le M\right\},
}
where
\eq{
X=H^1(0,T;L^{\infty}(\Omega\times V)\cap H^2(0,T;L^2(\Omega\times V)).
}
Here, $H^1,H^2$ denote usual Sobolev spaces over the specified domains.
A solution which satisfies 
(\ref{1.1}) - (\ref{rte0bc}) can be obtained in $\mathcal{U}$ due 
to the regularity and compatibility conditions of the initial value $a$ and 
the boundary value $g$.  As for the 
direct problems, see Bardos \cite{Bardos70}, Douglis \cite{Douglis66},
Prilepko and Ivankov \cite{Prilepko-Ivankov84}, and Ukai \cite{Ukai86}.

The study of inverse transport problems has started in radiative transfer by 
Bellman, Kagiwada, Kalaba and Ueno
\cite{Bellman-Kagiwada-Kalaba-Ueno65} and neutron transport by Case 
\cite{Case73}. 
As for application aspects of related inverse problems, see 
Arridge \cite{Arridge99} and Arridge and Schotland 
\cite {Arridge-Schotland09}.

For the mathematical analysis for inverse problems 
for radiative transport equations, we can have two main 
methodologies by (i) albedo operators and (ii) Carleman estimate.
Limited to the non-stationary case and not aiming at any comprehensive 
literature, we refer to works below.

First, the albedo operator can be interpreted as a mapping from boundary 
input on some subboundary to boundary data of the solution on other part of
the boundary. As for the approach by the albedo operator, we first 
refer to a review article by Bal \cite{Bal09}.   
Moreover, the uniqueness was studied by Choulli and Stefanov \cite{Choulli96}
and Stefanov \cite{Stefanov03}. 
In general, one can prove the stability of H\"older type. See
Bal and Jollivet \cite{Bal-Jollivet08,Bal-Jollivet09,Bal-Jollivet10}. 
This approach does not require strong assumptions such as nonzero initial 
values, but measurements have to be performed infinitely many times for 
obtaining the uniqueness and the stability.

Second, as for the approach by Carleman estimate, we refer 
to Bukhgeim and Klibanov \cite{Bukhgeim-Klibanov81} and 
Klibanov \cite{Klibanov84, Klibanov92} as pioneering works,
which apply Carleman estimates to inverse problems for second-order partial 
differential equations such as hyperbolic equations.  Such an approach 
yields the uniqueness and the stability for inverse problems for 
partial differential equations with a single measurement. 
Moreover, as for the inverse problems by Carleman estimates, one can consult 
Beilina and Klibanov \cite{Beilina-Klibanov12}, 
Bellassoued and Yamamoto \cite{BY2017},
Imanuvilov and Yamamoto \cite{Imanuvilov-Yamamoto01a,Imanuvilov-Yamamoto01b},
Isakov \cite{Isakov06}, Klibanov \cite{Klibanov13},
Klibanov and Timonov \cite{Klibanov-Timonov04}, and Yamamoto \cite{Yamamoto09}.

With the Carleman-estimate technique for the radiative transport 
equations, Klibanov and Pamyatnykh \cite{Klibanov-Pamyatnykh08}
proved the Lipschitz stability in determining $\sigma$ provided that 
$a(x,v) := u(x,v,0) \ne 0$ for all $(x,v) \in \overline{\Omega}
\times \overline{V}$ and 
\be
(a(x,v)\sigma(x,v))^2=(a(x,-v)\sigma(x,-v))^2 \quad 
\mbox {for $x\in\Omega$ and $v\in V$}
\label{1.4}
\ee
in the case of $V=\{v;\;|v|=1\}$. 
The extra condition (\ref{1.4}) is required because 
in \cite{Klibanov-Pamyatnykh08},
the extension of $u$ to the time interval $(-T,T)$ is necessary for the 
Carleman estimate, and (\ref{1.4}) is essential for the regularity of the extension.
The condition (\ref{1.4}) is concerned also with unknown $\sigma$, and so restrictive.
See also Klibanov and Pamyatnykh \cite{Klibanov-Pamyatnykh06} and 
Klibanov and Yamamoto \cite{Klibanov-Yamamoto07} as for related problems
on a radiative transport equation. 

After \cite{Klibanov-Pamyatnykh08}, we refer to Machida and 
Yamamoto \cite{Machida-Yamamoto14}:
it established a Carleman estimate with a linear weight function which is 
different from \cite{Klibanov-Pamyatnykh08} and the global Lipschitz 
stability for the inverse problems without any extension 
of the solution $u$ to $(-T,T)$.  In particular, any extra conditions for 
$\sigma(x,v)$ such as (\ref{1.4}) are not required.  
However, it must be assumed in \cite{Machida-Yamamoto14} that 
$V$ is a sectional domain, which means that $v$ is confined in narrow 
directions. More precisely, $v\in V$ must satisfy 
\be
(\gamma\cdot v)>0 \quad \mbox{for an arbitrary fixed $\gamma\in\Rm^n$}.
\label{1.5}
\ee

The main purpose of this article is to remove (\ref{1.5}) and improve 
\cite{Machida-Yamamoto14}.
More precisely, we prove the global Lipschitz stability results 
for the inverse problems with any bounded domain $V$ with 
$0\notin\overline{V}$, not necessarily satisfying (\ref{1.5}).

In this article, the weight function for the Carleman estimate is linear 
in $x,t$, similar to \cite{Machida-Yamamoto14}, but the main difference is
that we make choices of the weight according to suitably partitioned subomains 
of $V$ for deleting (\ref{1.5}), so that the weight function can be understood 
as piecewise linear function in $v$.

As for similar inverse problems for transport equations with 
$k\equiv 0$ in (\ref{1.1}), we refer to Gaitan-Ouzzane \cite{Gaitan-Ouzzane13}. Moreover one can consult Cannarsa, Floridia, G\"olgeleyen and Yamamoto 
\cite{CFGY}, Cannarsa, Floridia and Yamamoto \cite{Cannarsa-Floridia-Yamamoto} and G\"olgeleyen and Yamamoto \cite{Golgeleyen-Yamamoto16}, where a linear weight function is used.

The remainder of the paper is organized as follows. We state our main results 
in Section 2. 
The inverse coefficient problem reduces to an inverse source problem. 
Section 3 is devoted to the introduction of coupled radiative transport equations. In Section 4, we prove our key Carleman estimate. 
The energy estimate for the coupled radiative transport equations is 
established in Section 5.  The proof for the main theorem is given in 
Section 6.  Finally, we give concluding remarks in Section 7.

\section{Main results}
\label{main}

Let us choose $\widetilde{V}=\{v_0<|v|<v_1\}$ with sufficiently small $v_0$ and large $v_1$ such that $V\subset\widetilde{V}$. Then we take the zero extension for $\sigma(x,v)$ in $v$ and $k(x,v,v')$ in $(v,v')$, i.e., $\sigma=0$ for $v\in\widetilde{V}\setminus\overline{V}$, and $k=0$ if $v\in\widetilde{V}\setminus\overline{V}$ or $v'\in\widetilde{V}\setminus\overline{V}$. Moreover we set $a=g=0$ for $v\in\widetilde{V}\setminus\overline{V}$. Then we can replace $V$ in (\ref{1.1}), (\ref{rte0ic}), and (\ref{rte0bc}) with $\widetilde{V}$. The integral term can be further expressed as
\eq{
\int_Vk(x,v,v')u(x,v',t)\,dv'
&=\int_{\widetilde{V}}k(x,v,v')u(x,v',t)\,dv'
\\
&=
\sum_{j=0}^{m-1}\int_{V_j}k(x,v,v')u(x,v',t)\,dv',
}
where $m$ subdomains $V_j$ ($j=0,\dots,m-1$) are set as follows. We note that in spherical coordinates $v\in\tilde{V}$ is specified by $(r,\theta,\phi_1,\dots,\phi_{n-2})$, where $v_0<r<v_1$, $0\le\theta<2\pi$, $0\le\phi_i\le\pi$ ($i=1,\dots,n-2$) ($(r,\theta)$ in the case of $n=2$). We define subdomains $V_j$ ($j=0,\dots,m-1$) as
$$
V_j=\Bigl\{v\in\widetilde{V};
v_0<r<v_1,\;\frac{2\pi(l_0-1)}{L_0}\le\theta<\frac{2\pi l_0}{L_0},\;
$$
\begin{equation}\label{2.1}
\frac{\pi(l_i-1)}{L_i}\le\phi_i<\frac{\pi l_i}{L_i},\;
\;i=1,\dots,n-2\Bigr\},
\end{equation}
where $l_0=1,\dots,L_0$, $l_i=1,\dots,L_i$ ($i=1,\dots,n-2$), $m=L_0L_1\dots L_{n-2}$ and
\eq{
j=l_0+(l_1-1)L_0+(l_2-1)L_0L_1+\cdots+(l_{n-2}-1)L_0\dots L_{n-3}.
}
Let $\gamma_j\in V_j$ be arbitrarily chosen vectors ($j=0,1,\dots,m-1$). We take sufficiently large $m$ such that if $v\in V_j$, then $\gamma_j\cdot v\ge\kappa$ for an arbitrary fixed constant $\kappa>0$. That is, we have
\eq{
\min_{v\in\overline{V_j}}(\gamma_j\cdot v)>0,\quad j=0,\dots,m-1.
}

Our main results are stated as follows.

\begin{thm}
\label{mainthm}
Let $u^i(x,v,t)$ ($i=1,2$) be solutions to the radiative transport equation for $\sigma^i,a^i$, i.e.,
\eq{
\left\{\begin{aligned}
\left(\pp_t+v\cdot\nabla+\sigma^i(x,v)\right)u(x,v,t)-
\int_Vk(x,v,v')u(x,v',t)\,dv'=0,
\\
x\in\Omega,\;v\in V,\;0<t<T,
\\
u(x,v,0)=a^i(x,v),\quad x\in\Omega,\;v\in V,
\\
u(x,v,t)=g(x,v,t),\quad(x,v)\in\Gamma_-,\;0<t<T.
\end{aligned}\right.
}
Let $u^i\in\mathcal{U}$ and assume $\|\sigma^i\|_{L^{\infty}(\Omega\times V)}$, $\|k\|_{L^{\infty}(\Omega\times V\times V)}\le M$. Suppose that $T$ is large 
enough to satisfy
\al{
T>
\frac{\max_{0\le j\le m-1}\max_{x\in\overline{\Omega}}(\gamma_j\cdot x)-\min_{0\le j\le m-1}\min_{x\in\overline{\Omega}}(\gamma_j\cdot x)}
{\min_{0\le j\le m-1}\min_{v\in\overline{V_j}}(\gamma_j\cdot v)}.
\label{largeT}
}
We assume that there exists a constant $a_0>0$ such that
\al{
a^1(x,v)\ge a_0\quad\mbox{or}\quad a^2(x,v)\ge a_0,\quad
\mbox{a.e.~in}\;(x,v)\in\Omega\times V.
\label{nonzeroinit}
}
Then there exists a constant $C=C(M,a_0)>0$ such that
\eq{
\|\sigma^1-\sigma^2\|_{L^2(\Omega\times V)}
&\le
C\left(\int_0^T\int_{\Gamma_+}(\nu(x)\cdot v)
\left|\pp_t(u^1-u^2)(x,v,t)\right|^2\,dSdvdt\right)^{1/2}
\\
&+
C\left(\|a^1-a^2\|_{L^2(\Omega\times V)}+\|\nabla a^1-\nabla a^2\|_{L^2(\Omega\times V)}\right)
}
and
\eq{
&
\left(\int_0^T\int_{\Gamma_+}(\nu(x)\cdot v)|\pp_t(u^1-u^2)(x,v,t)|^2\,dSdvdt
\right)^{1/2}
\\
&\le
\|\sigma^1-\sigma^2\|_{L^2(\Omega\times V)}+\|a^1-a^2\|_{L^2(\Omega\times V)}+
\|\nabla a^1-\nabla a^2\|_{L^2(\Omega\times V)}.
}
Here we have $C(M,a_0)\to\infty$ as $M\to\infty$ or $a_0\to0$.
\end{thm}

Thus we have removed extra conditions (\ref{1.4}) and (\ref{1.5}) to prove the global Lipschitz stability for the inverse coefficient problem. Condition (\ref{largeT}) means that the critical length $T$ of the time interval depends on the partition of $V$.
 
\begin{cor}
If we assume $a^1=a^2$ in $\Omega\times V$, then we have the following both-sided estimate.
\eq{
&
\left(\int_0^T\int_{\Gamma_+}(\nu(x)\cdot v)|\pp_t(u^1-u^2)(x,v,t)|^2\,dSdvdt
\right)^{1/2}
\le
\|\sigma^1-\sigma^2\|_{L^2(\Omega\times V)}
\\
&\le
C\left(\int_0^T\int_{\Gamma_+}(\nu(x)\cdot v)
\left|\pp_t(u^1-u^2)(x,v,t)\right|^2\,dSdvdt\right)^{1/2}.
}
\end{cor}

This means that the choice of the norm for the boundary data $\Gamma_+\times(0,T)$ is the best possible for our inverse problem.

\begin{rmk}
Positive initial values in (\ref{nonzeroinit}) can be set up by 
a combination of some control procedure.  More precisely,
let us assume that $\Omega$ is strictly convex. Suppose $\sigma^2$ is known and we consider the radiative transport equation for $\sigma^2$ during the time interval $(-T_0,T)$ with some $T_0>0$. We extend $g$ such that $g$ belongs to some weighted $L^2$-space in $\Gamma_-\times(-T_0,0)$. The value $u^2(x,v,-T_0)$ may be either zero or nonzero. By the exact controllability result \cite{Klibanov-Yamamoto07}, we can have $u^2(x,v,0)=a^2(x,v)\ge a_0$, a.e.~in $(x,v)\in\Omega
\times V$, by adjusting the boundary value $g$ for sufficiently large $T_0>0$. 

In the case of optical tomography \cite{Arridge99,Arridge-Schotland09}, the boundary value $g$ is the incident laser beam of near-infrared light. We can prepare positive initial values by turning on the laser at $t=-T_0$ before starting to detect the out-going light on the surface of biological tissue at $t=0$.
\end{rmk}

\begin{rmk}
It is also possible to determine the scattering coefficient $\sigma_s$ if we write $k$ as $k(x,v,v')=\sigma_s(x,v)p(x,v,v')$. The proof is similar to that of Theorem \ref{mainthm}. We refer the reader to \cite{Machida-Yamamoto14}.
\end{rmk}

Next we state the second main result for an inverse source problem.
We consider the following radiative transport equation with an internal source term $F(x,v,t)$.
\al{
\label{subtracted}
\left\{\begin{aligned}
\left(\pp_t+v\cdot\nabla+\sigma(x,v)\right)u(x,v,t)
-\int_Vk(x,v,v')u(x,v',t)\,dv'=F(x,v,t),
\\
x\in\Omega,\; v\in V,\; 0<t<T,
\\
u(x,v,0)=a(x,v),
\quad x\in\Omega,\;v\in V,
\\
u(x,v,t)=g(x,v,t)
\quad\mbox{on}\;\Gamma_-\times(0,T).
\end{aligned}\right.
}
Let us assume that $F(x,v,t)$ has the following form.
\eq{
F(x,v,t)=f(x,v)R(x,v,t),
}
By subtraction for the equations in Theorem \ref{mainthm}, we obtain (\ref{subtracted}) with
\al{
u(x,v,t)=u^1(x,v,t)-u^2(x,v,t),\quad \sigma(x,v)=\sigma^1(x,v),
\label{subtr1}
}
\al{
a(x,v)=a^1(x,v)-a^2(x,v),\quad g(x,v,t)=0,
\label{subtr2}
}
\al{
f(x,v)=\sigma^1(x,v)-\sigma^2(x,v),\quad R(x,v,t)=-u^2(x,v,t).
\label{subtr3}
}

The following global Lipschitz stability is obtained for the inverse source problem for (\ref{subtracted}).

\begin{thm}
\label{invsource}
Let $u(x,v,t)$ be the solution to (\ref{subtracted}). Suppose $u\in\mathcal{U}$. We assume that $\|\sigma\|_{L^{\infty}(\Omega\times V)}$, $\|k\|_{L^{\infty}(\Omega\times V\times V)}\le M$. Suppose that $T$ satisfies (\ref{largeT}). We assume $R,\pp_tR\in L^2(0,T;L^{\infty}(\Omega\times V))$. For an arbitrary fixed constant $a_0>0$, we further assume that $R(x,v,0)>a_0$ almost all $(x,v)\in\Omega\times V$. Then there exists a constant $C=C(M,a_0)>0$ such that
\eq{
\|f\|_{L^2(\Omega\times V)}
&\le
C\left(\int_0^T\int_{\pp\Omega}\int_V(\nu\cdot v)|\pp_tu|^2\,dSdvdt\right)^{1/2}
\\
&+
C\left(\|a\|_{L^2(\Omega\times V)}+\|\nabla a\|_{L^2(\Omega\times V)}\right)
}
and
\eq{
&
\left(\int_0^T\int_{\pp\Omega}\int_V(\nu\cdot v)|\pp_tu|^2\,dSdvdt\right)^{1/2}
\\
&\le
\|f\|_{L^2(\Omega\times V)}+\|a\|_{L^2(\Omega\times V)}+
\|\nabla a\|_{L^2(\Omega\times V)}
}
for any $f\in L^2(\Omega\times V)$. If $g=0$, we have
\al{
\|f\|_{L^2(\Omega\times V)}
&\le
C\left(\int_0^T\int_{\Gamma_+}(\nu\cdot v)|\pp_tu|^2\,dSdvdt\right)^{1/2}
\nonumber \\
&+
C\left(\|a\|_{L^2(\Omega\times V)}+\|\nabla a\|_{L^2(\Omega\times V)}\right)
\label{invsource1}
}
and
\al{
&
\left(\int_0^T\int_{\Gamma_+}(\nu\cdot v)|\pp_tu|^2\,dSdvdt\right)^{1/2}
\nonumber \\
&\le
\|f\|_{L^2(\Omega\times V)}+\|a\|_{L^2(\Omega\times V)}+
\|\nabla a\|_{L^2(\Omega\times V)}
\label{invsource2}
}
for any $f\in L^2(\Omega\times V)$.
\end{thm}

The proof of Theorem 2.1 is reduced to Theorem 2.5 by 
substituting (\ref{subtr1}) - (\ref{subtr3}) in (\ref{invsource1}) 
and (\ref{invsource2}).  Thus it suffices to prove Theorem \ref{invsource}.
The proof of Theorem \ref{invsource} relies on Lemma \ref{carlemans}, which 
is a Carleman estimate proved in Section 4.

\section{Coupled radiative transport equations}
\label{coupled}

We recall (\ref{2.1}).
Then we can construct mappings $\mathcal{R}_j:V_0\to V_j$ ($j=1,\dots,m-1$) as
\eq{
\mathcal{R}_jv
&=\mathcal{R}_j(r,\theta,\phi_1,\dots,\phi_{n-2})
\\
&=
\left(r,\theta+\frac{2\pi(l_0-1)}{L_0},\phi_1+\frac{\pi(l_1-1)}{L_1},\dots,
\phi_{n-2}+\frac{\pi(l_{n-2}-1)}{L_{n-2}}\right).
}
We define $\mathcal{R}_0=1$. Then for $v\in V_0$,
\eq{
\mathcal{R}_1v\in V_1,\quad\mathcal{R}_2v\in V_2,\dots,\quad
\mathcal{R}_{m-1}v\in V_{m-1},\quad\mathcal{R}_mv\in V_0.
}

Let us define
\begin{equation}\label{3.1}
u_i(x,v,t)=u(x,\mathcal{R}_iv,t),\quad v\in V_0,\quad i=0,\dots,m-1.
\end{equation}
We introduce
\eq{
P_iu_i(x,v,t)=\pp_tu_i(x,v,t)+w_i\cdot\nabla u_i(x,v,t),
}
where
\[
w_i=\mathcal{R}_iv.
\]
Let us define
\eq{
\Gamma_{\pm}^i=\{(x,v)\in\partial\Omega\times V_0;\;
(\pm\nu(x)\cdot\mathcal{R}_iv) > 0\}.
}
Furthermore we define $a_i(x,v)=a(x,\mathcal{R}_iv)$, $g_i(x,v,t)=g(x,\mathcal{R}_iv,t)$, and
\eq{
\sigma_i(x,v)=\sigma(x,\mathcal{R}_iv),\quad
k_{ij}(x,v,v')=k(x,\mathcal{R}_iv,\mathcal{R}_jv),\quad
x\in\Omega,\quad v,v'\in V_0.
}
Thus (\ref{subtracted}) can be rewritten as
\al{
\label{general}
\left\{\begin{aligned}
&
P_iu_i(x,v,t)+\sigma_i(x,v)u_i(x,v,t)-
\int_{V_0}\sum_{j=0}^{m-1}k_{ij}(x,v,v')u_j(x,v',t)\,dv'=F_i(x,v,t),
\\
&\quad x\in\Omega,\; v\in V_0,\; 0<t<T,
\\
&
u_i(x,v,0)=a_i(x,v),
\quad x\in\Omega,\;v\in V_0,
\\
&
u_i(x,v,t)=g_i(x,v,t)
\quad\mbox{on}\;\Gamma_-^i\times(0,T)
\end{aligned}\right.
}
for $i=0,1,\dots,m-1$. Here we defined
\eq{
F_i(x,v,t)=f_i(x,v)R_i(x,v,t),\quad
f_i(x,v)=f(x,\mathcal{R}_iv),\quad
R_i(x,v,t)=R(x,\mathcal{R}_iv,t).
}

\section{Key Carleman estimate for coupled equations}
\label{carleman}

Let us introduce weight functions as
\eq{
\va_j(x,t)=(\gamma_j\cdot x)-\beta t,
\quad j=0,1,\dots,m-1,
}
where
\al{
0<\beta<\min_{0\le j\le m-1}\min_{v\in\overline{V_j}}(\gamma_j\cdot v).
\label{beta}
}
We set
\eq{
\mathcal{P}_iu_i(x,v,t)=P_iu_i(x,v,t)+\sigma_i(x,v)u_i(x,v,t)-
\int_{V_0}\sum_{j=0}^{m-1}k_{ij}(x,v,v')u_j(x,v',t)\,dv'
}
for $i=0,1,\dots,m-1$. Then the following inequality holds.

\begin{prop}
\label{carlemans}
We assume that $\sigma_{ij}\in L^{\infty}(\Omega\times V_0)$, $k_{ij}\in L^{\infty}(\Omega\times V_0\times V_0)$ ($i,j=0,\dots,m-1$). Furthermore we assume $u_i\in H^1(0,T;L^2(\Omega\times V_0))$, $\nabla u_i\in L^2(\Omega\times V_0\times(0,T))$, $i=0,\dots,m-1$. Suppose $u_i(\cdot,\cdot,T)=0$ in $\overline{\Omega\times V_0}$. Then there exist constants $s_0>0$ and $C>0$ such that
\al{
&
s\int_{V_0}\int_{\Omega}\sum_{i=0}^{m-1}|u_i(x,v,0)|^2e^{2s\va_i(x,0)}\,dxdv+
s^2\int_0^T\int_{V_0}\int_{\Omega}\sum_{i=0}^{m-1}|u_i|^2e^{2s\va_i}\,dxdvdt
\nonumber \\
&\le
C\int_0^T\int_{V_0}\int_{\Omega}\sum_{i=0}^{m-1}|\mathcal{P}_iu_i|^2e^{2s\va_i}
\,dxdvdt+
s\int_0^T\sum_{i=0}^{m-1}\int_{\Gamma_+^i}(w_i\cdot \nu)|u_i|^2e^{2s\va_i}
\,dSdvdt
\label{carleq}
}
for $s\ge s_0$.
\end{prop}

Proposition \ref{carlemans} is proved using Lemma \ref{carleman0} below. 
For a fixed $j \in \{ 0,\dots,m-1\}$, we define
\eq{
\widetilde{\mathcal{P}}_j=P_0u+\sigma u-\int_{V_j}k(x,v,v')u(x,v',t)\,dv',\quad
x\in\Omega,\;v\in V_j,\;t\in(0,T),
}
where $\sigma\in L^{\infty}(\Omega\times V_j)$, $k\in L^{\infty}(\Omega\times V_j\times V_j)$. Furthermore we set
\eq{
Q=\Omega\times(0,T).
}
The following Carleman estimate is obtained in \cite{Machida-Yamamoto14}.

\begin{lem}
\label{carleman0}
For a fixed $j \in \{0,\dots,m-1\}$, 
there exist constants $s_0>0$ and $C>0$ such that
\eq{
&
s\int_{V_j}\int_{\Omega}|u(x,v,0)|^2e^{2s\va_j(x,0)}\,dxdv+
s^2\int_Q\int_{V_j}|u(x,v,t)|^2e^{2s\va_j(x,v)}\,dxdvdt
\\
&\le
C\int_Q\int_{V_j}|\widetilde{\mathcal{P}}_ju|^2e^{2s\va_j(x,v)}\,dxdvdt
+ s\int_0^T\int_{V_j}\int_{\pp\Omega \cap \{(v\cdot\nu)>0\}}
(v\cdot \nu)|u|^2e^{2s\va_j(x,v)}\,dSdvdt
}
for all $s\ge s_0$ and $u\in H^1(0,T;L^2(\Omega\times V))$ satisfying $\nabla u\in L^2(\Omega\times V\times (0,T))$ and $u(\cdot,\cdot,T)=0$ in $\Omega\times V_j$.
\end{lem}

\begin{proof}
In the proof below we write $\va(x,t)=\va_j(x,t)$. See also \cite{Machida-Yamamoto14}.

Let us set $z(x,v,t)=e^{s\va(x,t)}u(x,v,t)$ and $Lz(x,v,t)=e^{s\va(x,t)}P_0(e^{-s\va(x,t)}z(x,v,t))$. That is, we have
\eq{
Lz(x,v,t)=P_0z(x,v,t)-sBz(x,v,t),
}
where
\eq{
B=(\pp_t+v\nabla)\va(x,t)=-\beta+v\cdot\gamma_j>0.
}
We note that
\eq{
\int_Q|P_0u(x,v,t)|^2e^{2s\va(x,t)}\,dxdt=\int_Q|Lz(x,v,t)|^2\,dxdt,\quad
v\in V_j.
}
The following calculation holds for almost all $v\in V_j$.
\eq{
&
\int_Q|P_0u|^2e^{2s\va}\,dxdt
\\
&=
\int_Q|\pp_tz+v\cdot\nabla z|^2\,dxdt+\int_Q|sB|^2z^2\,dxdt
-2s\int_QBz(\pp_tz+v\cdot\nabla z)\,dxdt
\\
&\ge
-2s\int_QBz(\pp_tz+v\cdot\nabla z)\,dxdt+s^2\int_QB^2z^2\,dxdt
\\
&=
-sB\int_Q\left(\pp_t(z^2)+v\cdot\nabla(z^2)\right)\,dxdt+s^2B^2\int_Qz^2\,dxdt
\\
&=
sB\int_{\Omega}|z(x,v,0)|^2\,dx-sB\int_0^T\int_{\pp\Omega}(v\cdot\nu)z^2\,dSdt
+s^2B^2\int_Qz^2\,dxdt
\\
&\ge
sB\int_{\Omega}|z(x,v,0)|^2\,dx
- sB\int_0^T\int_{\pp\Omega\cap\{(v\cdot\nu(x))>0\}}(v\cdot\nu)z^2\,dSdt
+s^2B^2\int_Qz^2\,dxdt.
}
By substituting $z=e^{s\va}u$ and integrating over $v$ in the above inequality, we obtain 
\eq{
&
C\int_Q\int_{V_j}|P_0u|^2e^{2s\va}\,dvdxdt
\ge
s\int_{\Omega}\int_{V_j}|u(x,v,0)|^2e^{2s\va(x,0)}\,dx
\\
&-
s\int_0^T\int_{\pp\Omega\cap \{(v\cdot\nu(x))>0\}}\int_{V_j}(v\cdot\nu)
|u|^2e^{2s\va}\,dvdSdt
+s^2\int_Q\int_{V_j}|u|^2e^{2s\va}\,dvdxdt.
}
It is straightforward to replace $|P_0u|^2$ in the above inequality with $|\widetilde{\mathcal{P}}_ju|^2$ and the proof is complete.
\end{proof}

We can rewrite the inequality in Lemma \ref{carleman0} as
\eq{
&
s\int_{V_0}\int_{\Omega}|u_j(x,v,0)|^2e^{2s\va_j(x,0)}\,dxdv+
s^2\int_Q\int_{V_0}|u_j(x,v,t)|^2e^{2s\va_j(x,v)}\,dxdvdt
\\
&\le
C\int_Q\int_{V_0}|\mathcal{P}_ju_j|^2e^{2s\va_j(x,v)}\,dxdvdt+
s\int_0^T\int_{\Gamma_+^j}(v\cdot \nu)|u_j|^2e^{2s\va_j(x,v)}\,dSdvdt.
}
By summing up the inequalities from $j=0$ through $j=m-1$, we obtain 
the inequality (\ref{carleq}). Thus Proposition \ref{carlemans} is proved.

\section{Energy estimates}
\label{energy}

Henceforth in this section, $C>0$ denotes generic constants which 
are independent of $f_i$.

\begin{lem}
\label{energy1}
The following inequalities hold for the solutions $u_i(x,v,t)$,
$i=0,\dots,m-1$, which satisfy (\ref{general}).
\al{
&
\int_{\Omega}\int_{V_0}\sum_{i=0}^{m-1}|\pp_tu_i(x,v,t)|^2\,dvdx
\nonumber \\
&\le
C\sum_{i=0}^{m-1}\left(\|f_i\|_{L^2(\Omega\times V_0)}^2+
\|a_i\|_{L^2(\Omega\times V_0)}^2+
\|\nabla a_i\|_{L^2(\Omega\times V_0)}^2\right)
\nonumber \\
&+
C\int_0^T\sum_{i=0}^{m-1}\int_{\Gamma_-^i}|(w_i\cdot\nu)||\pp_tu_i|^2\,dSdvdt
\label{eneq1}
}
and
\al{
&
\int_0^T\sum_{i=0}^{m-1}\int_{\Gamma_+^i}(w_i\cdot\nu)|\pp_tu_i|^2\,dSdvdt
\le
\int_0^T\sum_{i=0}^{m-1}\int_{\Gamma_-^i}|(w_i\cdot\nu)||\pp_tu_i|^2\,dSdvdt
\nonumber \\
&+
\sum_{i=0}^{m-1}\left(\|f_i\|_{L^2(\Omega\times V_0)}^2+
\|a_i\|_{L^2(\Omega\times V_0)}^2+
\|\nabla a_i\|_{L^2(\Omega\times V_0)}^2\right).
\label{eneq2}
}
\end{lem}

\begin{proof}

We differentiate the coupled transport equation in (\ref{general}) with respect to $t$ and obtain
\eq{
\pp_t(\pp_tu_i)+w_i\cdot\nabla(\pp_tu_i)+\sigma_i(\pp_tu_i)
-\int_{V_0}\sum_{j=0}^{m-1}k_{ij}(\pp_tu_j)\,dv'=
f_i\pp_tR_i.
}
By multiplying $2\pp_tu_i$, we have
\eq{
\pp_t(\pp_tu_i)^2+w_i\cdot\nabla (\pp_tu_i)^2+2\sigma_i(\pp_tu_i)^2-
2(\pp_tu_i)\int_{V_0}\sum_{j=0}^{m-1}k_{ij}(\pp_tu_j(x,v',t))\,dv'
\\
=2(\pp_tu_i)f_i\pp_tR_i.
}
By integrating over $\Omega\times V_0$, we obtain
\eq{
&
\pp_t\int_{\Omega}\int_{V_0}|\pp_tu_i|^2\,dvdx
+\int_{\Omega}\int_{V_0} w_i\cdot\nabla(|\pp_tu_i|^2)\,dvdx
+2\int_{\Omega}\int_{V_0}\sigma_i|\pp_tu_i|^2\,dvdx
\\
&-
2\int_{\Omega}\int_{V_0}(\pp_tu_i(x,v,t))\left(\int_{V_0}\sum_{j=0}^{m-1}
k_{ij}(x,v,v')\pp_tu_j(x,v',t)\,dv'\right)\,dvdx
\\
&=
2\int_{\Omega}\int_{V_0}(\pp_tu_i)f_i(\pp_tR_i)\,dvdx.
}
Setting $E(t)=\int_{\Omega}\int_{V_0}\sum_{i=0}^{m-1}|\pp_tu_i(x,v,t)|^2\,dvdx$
and integrating the second term on the left-hand side, we obtain
\eq{
&
\pp_tE(t)=
-\int_{\pp\Omega}\int_{V_0}\sum_{i=0}^{m-1}(w_i\cdot\nu)|\pp_tu_i|^2\,dvdS
-2\int_{\Omega}\int_{V_0}\sum_{i=0}^{m-1}\sigma_i|\pp_tu_i|^2\,dvdx
\\
&+
2\int_{\Omega}\int_{V_0}\sum_{i=0}^{m-1}(\pp_tu_i(x,v,t))
\left(\int_{V_0}\sum_{j=0}^{m-1}k_{ij}(\pp_tu_j(x,v',t))\,dv'\right)\,dvdx
\\
&+
2\int_{\Omega}\int_{V_0}\sum_{i=0}^{m-1}(\pp_tu_i)f_i(\pp_tR_i)\,dvdx.
}
Then by integrating over $t$, we have
\eq{
&
E(t)-E(0)
\\
&=
-\int_0^t\sum_{i=0}^{m-1}\left(\int_{\Gamma_+^i}+\int_{\Gamma_-^i}\right)
(w_i\cdot\nu)|\pp_tu_i|^2\,dSdvdt
-2\int_0^t\int_{\Omega}\int_{V_0}\sum_{i=0}^{m-1}\sigma_i|\pp_tu_i|^2\,dvdxdt
\\
&+
2\int_0^t\int_{\Omega}\int_{V_0}\sum_{i=0}^{m-1}(\pp_tu_i)
\left(\int_{V_0}\sum_{j=0}^{m-1}k_{ij}(\pp_tu_j(x,v',t))\,dv'\right)\,dvdxdt
\\
&+
2\int_0^t\int_{\Omega}\int_{V_0}\sum_{i=0}^{m-1}(\pp_tu_i)f_i(\pp_tR_i)
\,dvdxdt.
}
We note that by the Cauchy-Schwarz inequality,
\eq{
&
\left|\int_{V_0}\sum_{i=0}^{m-1}\left(\pp_tu_i(x,v,t)\right)
\left(\int_{V_0}\sum_{j=0}^{m-1}k_{ij}(x,v,v')\left(\pp_tu_j(x,v',t)\right)
\,dv'\right)\,dv\right|
\\
&\le
C\int_{V_0}\sum_{i=0}^{m-1}\left|\pp_tu_i(x,v,t)\right|
\left(\int_{V_0}\sum_{j=0}^{m-1}\left|\pp_tu_j(x,v',t)\right|\,dv'\right)\,dv
\\
&\le
C\sum_{i=0}^{m-1}\sum_{j=0}^{m-1}\left(\int_{V_0}\left|\pp_tu_i(x,v,t)\right|^2\,dv\right)^{1/2}
|V_0|^{1/2}\left(\int_{V_0}\left|\pp_tu_j(x,v',t)\right|^2\,dv'\right)^{1/2}|V_0|^{1/2}
\\
&\le
C|V_0|\int_{V_0}\sum_{i=0}^{m-1}\left|\pp_tu_i(x,v,t)\right|^2\,dv,
}
where $|V_0|=\int_{V_0}dv$, and
\eq{
2\int_{\Omega}\int_{V_0}|(\pp_tu_i)f_i(\pp_tR_i)|\,dvdx
&\le
\int_{\Omega}\int_{V_0}|f_i(\pp_tR_i)|^2\,dvdx
+\int_{\Omega}\int_{V_0}|\pp_tu_i|^2\,dvdx
\\
&\le
C\int_{\Omega}\int_{V_0}|f_i|^2\,dvdx+
\int_{\Omega}\int_{V_0}|\pp_tu_i|^2\,dvdx.
}
Hence we obtain for $0\le t\le T$,
\al{
E(t)-E(0)
&\le
-\int_0^t\sum_{i=0}^{m-1}\left(\int_{\Gamma_+^i}+\int_{\Gamma_-^i}\right)
(w_i\cdot\nu)|\pp_tu_i|^2\,dSdvdt
+C\int_0^tE(\eta)\,d\eta
\nonumber \\
&+
C\sum_{i=0}^{m-1}\|f_i\|_{L^2(\Omega\times V_0)}^2.
\label{ene2}
}
We note that from (\ref{general}),
\eq{
\pp_tu_i(x,v,0)+w_i\cdot\nabla a_i(x,v)+\sigma_ia_i(x,v)
-\int_{V_0}\sum_{j=0}^{m-1}k_{ij}a_j(x,v')\,dv'
=f_iR_i(x,v,0),
}
and hence
\eq{
E(0)\le C\sum_{i=0}^{m-1}\left(\|f_i\|_{L^2(\Omega\times V_0)}^2+
\|a_i\|_{L^2(\Omega\times V_0)}^2+\|\nabla a_i\|_{L^2(\Omega\times V_0)}^2
\right).
}
Using the Gronwall inequality, we arrive at
\al{
E(t)
&\le
C\sum_{i=0}^{m-1}\left(-\int_0^t\Biggl(\int_{\Gamma_+^i}+\int_{\Gamma_-^i}\right)(w_i\cdot\nu)|\pp_tu_i|^2\,dSdvdt
\nonumber \\
&+\|f_i\|_{L^2(\Omega\times V_0)}^2+\|a_i\|_{L^2(\Omega\times V_0)}^2
+\|\nabla a_i\|_{L^2(\Omega\times V_0)}^2\Biggr).
\label{ene4}
}
Noting that $\int^t_0\int_{\Gamma_+^i}(w_i\cdot\nu)|\pp_tu_i|^2\,dSdvdt\ge 0$, the first inequality (\ref{eneq1}) in Lemma \ref{energy1} is proved from (\ref{ene4}).

Equation (\ref{ene4}) yields
\eq{
&
\int_0^T\sum_{i=0}^{m-1}\int_{\Gamma_+^i}(w_i\cdot\nu)|\pp_tu_i|^2\,dSdvdt
\\
&\le
-\frac{E(T)}{C}-
\int_0^T\sum_{i=0}^{m-1}\int_{\Gamma_-^i}(w_i\cdot\nu)|\pp_tu_i|^2\,dSdvdt
\\
&+
\sum_{i=0}^{m-1}\left(\|f_i\|_{L^2(\Omega\times V_0)}^2+
\|a_i\|_{L^2(\Omega\times V_0)}^2+
\|\nabla a_i\|_{L^2(\Omega\times V_0)}^2\right)
\\
&\le
-\int_0^T\sum_{i=0}^{m-1}\int_{\Gamma_-^i}(w_i\cdot\nu)|\pp_tu_i|^2\,dSdvdt
\\
&+
\sum_{i=0}^{m-1}\left(\|f_i\|_{L^2(\Omega\times V_0)}^2+
\|a_i\|_{L^2(\Omega\times V_0)}^2+\|\nabla a_i\|_{L^2(\Omega\times V_0)}^2
\right).
}
Thus the second inequality (\ref{eneq2}) in Lemma \ref{energy1} is proved.
\end{proof}

\section{Proof of Theorem \ref{invsource}}
\label{proof}

Hereafter $C>0$ denotes generic constants which are independent of $s>0$.

We set
\eq{
r_{\rm max}=
\max_{0\le j\le m-1}\max_{x\in\overline{\Omega}}(\gamma_j\cdot x),\quad
r_{\rm min}=
\min_{0\le j\le m-1}\min_{x\in\overline{\Omega}}(\gamma_j\cdot x).
}
Since $T$ satisfies
\eq{
T>\frac{r_{\rm max}-r_{\rm min}}
{\min_{0\le j\le m-1}\min_{v\in\overline{V_j}}(\gamma_j\cdot v)},
}
We can choose $\beta$ such that (\ref{beta}) and
\eq{
r_{\rm max}-\beta T<r_{\rm min}.
}
Then we have
\eq{
\va_i(x,T)\le r_{\rm max}-\beta T<r_{\rm min}\le\va_i(x,0),\quad
i=0,\dots,m-1,
\; x\in\overline{\Omega}.
}
Therefore there exist $\delta>0$ and $r_0,r_1$ such that
$r_{\rm max}-\beta T<r_0<r_1<r_{\rm min}$,
\eq{
\va_i(x,t)>r_1,\quad x\in\overline{\Omega},\;0\le t\le\delta,
}
and
\eq{
\va_i(x,t)<r_0,\quad x\in\overline{\Omega},\;T-2\delta\le t\le T
}
for $i=0,1,\dots,m-1$. Let us introduce a cut-off function $\chi\in C^{\infty}_0({\Bbb R})$ such that $0\le\chi\le1$ and
\eq{
\chi(t)=\left\{\begin{aligned}
1,&\quad 0\le t\le T-2\delta,
\\
0,&\quad T-\delta\le t\le T.
\end{aligned}\right.
}

Let us set
\eq{
z_i(x,v,t)=(\pp_tu_i(x,v,t))\chi(t),\quad i=0,\dots,m-1.
}
Then we have $z_i(x,v,T)=0$. By differentiating the equation in (\ref{general}), we obtain
\al{
&
\tilde{P}_iz_i(x,v,t)-\int_{V_0}\sum_{j=0}^{m-1}k_{ij}(x,v,v')z_j(x,v',t)\,dv'=
\chi f_i(\pp_tR_i)+(\pp_t\chi)\pp_tu_i
\label{wrte0}
}
for $(x,t)\in Q$, $v\in V_0$, where
\eq{
\tilde{P}_iu_i(x,v,t)=P_iu_i(x,v,t)+\sigma_iu_i(x,v,t),\quad
(x,t)\in Q,\;v\in V_0.
}
From the coupled radiative transport equation in (\ref{general}), we have at $t=0$,
\al{
z_i(x,v,0)
&=
f_i(x,v)R_i(x,v,0)-w_i\cdot\nabla a_i(x,v)-\sigma_i(x,v)a_i(x,v)
\nonumber \\
&+
\int_{V_0}\sum_{j=0}^{m-1}k_{ij}(x,v,v')a_j(x,v')\,dv',
\quad x\in\Omega,\;v\in V_0.
\label{wk0}
}
Using the Carleman estimate in Theorem \ref{carlemans}, we obtain for $z_i$,
\al{
&
s\int_{V_0}\int_{\Omega}\sum_{i=0}^{m-1}|z_i(x,v,0)|^2e^{2s\va_i(x,0)}\,dxdv
\le
C\int_Q\int_{V_0}\sum_{i=0}^{m-1}\left|\chi f_i(\pp_tR_i)\right|^2e^{2s\va_i(x,t)}\,dvdxdt
\nonumber \\
&+
C\int_Q\int_{V_0}\sum_{i=0}^{m-1}\left|(\pp_t\chi)\pp_tu_i\right|^2
e^{2s\va_i(x,t)}\,dvdxdt
\nonumber \\
&+
Cs\int_0^T\sum_{i=0}^{m-1}\int_{\Gamma_+^i}(w_i\cdot\nu)|z_i|^2e^{2s\va_i(x,t)}
\,dSdvdt.
\label{carlemanw}
}
Let us set
\eq{
d_0=\left(\int^T_0\sum_{i=0}^{m-1}\int_{\Gamma_+^i}(w_i\cdot\nu)|\pp_tu_i|^2
\,dSdvdt\right)^{1/2}.
}
The last term in (\ref{carlemanw}) is estimated as
\eq{
Cs\int_0^T\sum_{i=0}^{m-1}\int_{\Gamma_+^i}(w_i\cdot\nu)|z_i|^2e^{2s\va_i(x,t)}
\,dSdvdt
\le Ce^{C_1s}d_0^2,
}
where we used $se^{Cs}\le e^{(C+1)s}$ for $s>0$ and set $C_1=C+1$. Since $\pp_t\chi=0$ for $0\le t\le T-2\delta$ or $T-\delta\le t \le T$, we have
\al{
\int_Q\int_{V_0}\sum_{i=0}^{m-1}|(\pp_t\chi)\pp_tu_i|^2e^{2s\va_i(x,t)}\,dvdxdt
&=
\int_{T-2\delta}^{T-\delta}\int_{\Omega}\int_{V_0}\sum_{i=0}^{m-1}|
(\pp_t\chi)\pp_tu_i|^2e^{2s\va_i(x,t)}\,dvdxdt
\nonumber \\
&\le
Ce^{2sr_0}\int_{T-2\delta}^{T-\delta}\int_{\Omega}\int_{V_0}\sum_{i=0}^{m-1}
|\pp_tu_i|^2\,dvdxdt.
\label{eq10}
}
Thus from (\ref{eq10}) with the help of (\ref{eneq1}) in Lemma \ref{energy1},
\eq{
&
\int_{\Omega}\int_{V_0}\sum_{i=0}^{m-1}|(\pp_t\chi)\pp_tu_i|^2e^{2s\va}\,dvdxdt
\\
&\le
Ce^{2sr_0}\sum_{i=0}^{m-1}\left(\|f_i\|_{L^2(\Omega\times V_0)}^2+
\|a_i\|_{L^2(\Omega\times V_0)}^2+
\|\nabla a_i\|_{L^2(\Omega\times V_0)}^2\right)
\\
&-
Ce^{2sr_0}\int_0^T\sum_{i=0}^{m-1}\int_{\Gamma_-^i}(w_i\cdot\nu)|\pp_tu_i|^2
\,dSdvdt.
}

Note that (\ref{wk0}) holds for $x\in\overline{\Omega}$, $v\in\overline{V_0}$, and recall the fact that $R_i(\cdot,\cdot,0)\ne 0$ in $\overline{\Omega\times V_0}$. We obtain
\eq{
&
\int_{\Omega}\int_{V_0}\sum_{i=0}^{m-1}|z_i(x,v,0)|^2e^{2s\va_i(x,0)}\,dvdx
+Ce^{Cs}\sum_{i=0}^{m-1}\left(\|a_i\|_{L^2(\Omega\times V_0)}^2+
\|\nabla a_i\|_{L^2(\Omega\times V_0)}^2\right)
\\
&\ge
C\int_{\Omega}\int_{V_0}\sum_{i=0}^{m-1}|f_i(x,v)|^2e^{2s\va_i(x,0)}\,dvdx.
}
Therefore (\ref{carlemanw}) yields
\eq{
&
s\int_{\Omega}\int_{V_0}\sum_{i=0}^{m-1}|f_i(x,v)|^2e^{2s\va_i(x,0)}\,dvdx
\le
C\int_Q\int_{V_0}\sum_{i=0}^{m-1}|f_i(x,v)|^2e^{2s\va_i(x,t)}\,dvdxdt
\\
&+
Ce^{2sr_0}\sum_{i=0}^{m-1}\|f_i\|_{L^2(\Omega\times V_0)}^2
-Ce^{2sr_0}\int_0^T\sum_{i=0}^{m-1}\int_{\Gamma_-^i}(w_i\cdot\nu)|\pp_tu_i|^2
\,dSdvdt
+Ce^{C_1s}d^2,
}
where we defined
\eq{
d=\left(\int_0^T\sum_{i=0}^{m-1}\int_{\Gamma_+^i}(w_i\cdot\nu)|\pp_tu_i|^2
\,dSdvdt\right)^{1/2}+
\sum_{i=0}^{m-1}\left(\|a_i\|_{L^2(\Omega\times V_0)}+\|\nabla a_i\|_{L^2(\Omega\times V_0)}\right).
}
Since $\va_i(x,t)\le\va_i(x,0)$ for $(x,t)\in Q$,
\eq{
&
(s-CT)\int_{\Omega}\int_{V_0}\sum_{i=0}^{m-1}|f_i(x,v)|^2e^{2s\va_i(x,0)}\,dvdx
\le
Ce^{2sr_0}\sum_{i=0}^{m-1}\|f_i\|_{L^2(\Omega\times V_0)}^2
\\
&-
Ce^{2sr_0}\int_0^T\sum_{i=0}^{m-1}\int_{\Gamma_-^i}(w_i\cdot\nu)|\pp_tu_i|^2
\,dSdvdt+Ce^{C_1s}d^2.
}
Noting that $\va_i(x,0)>r_1$, we have
\eq{
&
se^{2sr_1}\int_{\Omega}\int_{V_0}\sum_{i=0}^{m-1}|f_i|^2
\,dvdx
\le
Ce^{2sr_0}\sum_{i=0}^{m-1}\|f_i\|_{L^2(\Omega\times V_0)}^2
\\
&-
Ce^{Cs}\int_0^T\sum_{i=0}^{m-1}\int_{\Gamma_-^i}(w_i\cdot\nu)|\pp_tu_i|^2
\,dSdvdt
+Ce^{C_1s}d^2.
}
Hence, for sufficiently large $s$,
\al{
&
\sum_{i=0}^{m-1}\|f_i\|_{L^2(\Omega\times V_0)}^2
\nonumber \\
&\le
Ce^{-2s(r_1-r_0)}\sum_{i=0}^{m-1}\|f_i\|_{L^2(\Omega\times V_0)}^2-
Ce^{Cs}\int_0^T\sum_{i=0}^{m-1}\int_{\Gamma_-^i}(w_i\cdot\nu)|\pp_tu_i|^2
\,dSdvdt
\nonumber \\
&+
Ce^{Cs}\int_0^T\sum_{i=0}^{m-1}\int_{\Gamma_+^i}(w_i\cdot\nu)
|\pp_tu_i|^2\,dSdvdt
\nonumber \\
&+
Ce^{Cs}\sum_{i=0}^{m-1}\left(\|a_i\|_{L^2(\Omega\times V_0)}^2
+\|\nabla a_i\|_{L^2(\Omega\times V_0)}^2\right).
\label{almost}
}
The first term on the right-hand side of (\ref{almost}) will vanish as $s$ becomes large. We can rewrite (\ref{almost}) as
\eq{
\|f\|_{L^2(\Omega\times V)}^2
&\le
-Ce^{Cs}\int_0^T\int_{\pp\Omega}\int_V|(v\cdot\nu)||\pp_tu|^2\,dvdSdt
\\
&+
Ce^{Cs}\left(\|a\|_{L^2(\Omega\times V)}^2+
\|\nabla a\|_{L^2(\Omega\times V)}^2\right)
}
for sufficiently large $s$. Hence the first inequality in Theorem \ref{invsource} is proved.

The second inequality in Theorem \ref{invsource} is immediately obtained from (\ref{eneq2}) in Lemma \ref{energy1}. Thus the proof is complete.
\qed

\section{Concluding remarks}
\label{concl}

In \cite{Machida-Yamamoto14}, the velocity $v$ must satisfy 
$(v\cdot\gamma) > 0$ 
with some fixed vector $\gamma\in\Rm^n$, for which there are limited 
applications in transport phenomena. 

In the present article, we deleted such an extra assumption and our 
global stability results Theorems 2.1 and 2.5 require only 
the positivity (2.3) of initial values and $R(x,v,0) > 0$ on 
$(x,v) \in \overline{\Omega} \times \overline{V}$ respectively.

Moreover, it is understood that such positivity is essential for the 
methodology by Carleman estimate, in general.
  
With the partition by choosing multiple fixed vectors $\gamma_j$, 
$j=0,\dots,m-1$, which are dependent on $v$, we construct the 
weight functions in the form $\varphi_i(x,t) := (\gamma_j\cdot x) - \beta t$ 
to derive the key Carleman estimate Proposition 4.1.
Such dependence of the weight on $v$  still admits to prove the relevant 
Carleman estimate for $u$.

\section*{Acknowledgements}

The first author acknowledges support from Grant-in-Aid for Scientific Research (C) 17K05572, 18K03438 of Japan Society for the Promotion of Science. 
The second author was supported by Grant-in-Aid for Scientific Research (S)
15H05740 of Japan Society for the Promotion of Science and
by The National Natural Science Foundation of China
(no. 11771270, 91730303).
This work was prepared with the support of the "RUDN University Program 5-100".



\end{document}